\documentclass[12pt]{amsart}
\usepackage{amsfonts,amssymb,amscd,amsmath,enumerate,color, xypic}
\def\NZQ{\mathbb}               

\def\ZZ{{\NZQ Z}}
\def\RR{{\NZQ R}}

\def\frk{\mathfrak}               

\def\Phi{{\frk N}}
\def\ab{{\bold a}}

\def\eb{{\bold e}}
\def\tb{{\bold t}}

\def\opn#1#2{\def#1{\operatorname{#2}}} 
\opn\chara{char} 
\opn\length{\ell} 
\opn\pd{pd} 
\opn\rk{rk}
\opn\projdim{proj\,dim} 
\opn\injdim{inj\,dim} 
\opn\rank{rank}
\opn\depth{depth} 
\opn\grade{grade} 
\opn\height{height}
\opn\embdim{emb\,dim} 
\opn\codim{codim}

\opn\Tr{Tr} 
\opn\bigrank{big\,rank}
\opn\superheight{superheight}
\opn\lcm{lcm}
\opn\trdeg{tr\,deg}
\opn\reg{reg} 
\opn\lreg{lreg} 
\opn\ini{in} 
\opn\lpd{lpd}
\opn\size{size}
\opn\mult{mult}
\opn\dist{dist}
\opn\cone{cone}
\opn\lex{lex}
\opn\rev{rev}
\opn\div{div} \opn\Div{Div} \opn\cl{cl} \opn\Cl{Cl}
\opn\Spec{Spec} \opn\Supp{Supp} \opn\supp{supp} \opn\Sing{Sing}
\opn\Ass{Ass} \opn\Min{Min}
\opn\Ann{Ann} \opn\Rad{Rad} \opn\Soc{Soc}
\opn\Syz{Syz} \opn\Im{Im} \opn\Ker{Ker} \opn\Coker{Coker}
\opn\Am{Am} \opn\Hom{Hom} \opn\Tor{Tor} \opn\Ext{Ext}
\opn\End{End} \opn\Aut{Aut} \opn\id{id} \opn\ini{in}

\opn\nat{nat}
\opn\pff{pf}
\opn\Pf{Pf} \opn\GL{GL} \opn\SL{SL} \opn\mod{mod} \opn\ord{ord}
\opn\Gin{Gin}
\opn\Hilb{Hilb}\opn\adeg{adeg}\opn\std{std}\opn\ip{infpt}
\opn\Pol{Pol}
\opn\sat{sat}
\opn\Var{Var}
\opn\Gen{Gen}

\opn\aff{aff} \opn\con{conv} \opn\relint{relint} \opn\st{st}
\opn\lk{lk} \opn\cn{cn} \opn\core{core} \opn\vol{vol}
\opn\link{link} \opn\star{star}
\opn\gr{gr}

\def\Ac{{\mathcal A}}

\def\Jc{{\mathcal J}}
\def\Gc{{\mathcal G}}

\def\Oc{{\mathcal O}}
\def\Pc{{\mathcal P}}
\def\Qc{{\mathcal Q}}

\def\Cc{{\mathcal C}}

\def\pot#1#2{#1[\kern-0.28ex[#2]\kern-0.28ex]}

\opn\dirlim{\underrightarrow{\lim}}
\opn\inivlim{\underleftarrow{\lim}}
\let\to=\rightarrow

\def\Implies{\ifmmode\Longrightarrow \else
	\unskip${}\Longrightarrow{}$\ignorespaces\fi}
\def\implies{\ifmmode\Rightarrow \else
	\unskip${}\Rightarrow{}$\ignorespaces\fi}
\def\iff{\ifmmode\Longleftrightarrow \else
	\unskip${}\Longleftrightarrow{}$\ignorespaces\fi}

\let\:=\colon
\newtheorem{Theorem}{Theorem}[section]
\newtheorem{Lemma}[Theorem]{Lemma}

\newtheorem{Proposition}[Theorem]{Proposition}

\newtheorem{Example}[Theorem]{Example}

\newtheorem{Question}[Theorem]{Question}
\let\epsilon\varepsilon
\let\phi=\varphi
\let\kappa=\varkappa
\def\qed{\ifhmode\textqed\fi
	\ifmmode\ifinner\quad\qedsymbol\else\dispqed\fi\fi}
\def\textqed{\unskip\nobreak\penalty50
	\hskip2em\hbox{}\nobreak\hfil\qedsymbol
	\parfillskip=0pt \finalhyphendemerits=0}
\def\dispqed{\rlap{\qquad\qedsymbol}}

\opn\dis{dis}
\opn\height{height}
\opn\dist{dist}
\def\pnt{{\raise0.5mm\hbox{\large\bf.}}}

\opn\Lex{Lex}
\opn\conv{conv}

\begin{document}
\title{Facets and volume of Gorenstein Fano polytopes}
\author[T. Hibi]{Takayuki Hibi}
\address[Takayuki Hibi]{Department of Pure and Applied Mathematics,
	Graduate School of Information Science and Technology,
	Osaka University,
	Suita, Osaka 565-0871, Japan}
\email{hibi@math.sci.osaka-u.ac.jp}
\author[A. Tsuchiya]{Akiyoshi Tsuchiya}
\address[Akiyoshi Tsuchiya]{Department of Pure and Applied Mathematics,
	Graduate School of Information Science and Technology,
	Osaka University,
	Suita, Osaka 565-0871, Japan}
\email{a-tsuchiya@cr.math.sci.osaka-u.ac.jp}

\date{}
\thanks{The second author was partially supported by Grant-in-Aid for JSPS Fellows 16J01549.}

\begin{abstract}
		It is known that every integral convex polytope is unimodularly
	equivalent to a face of some Gorenstein Fano polytope. It is then
	reasonable to ask whether every normal polytope is unimodularly
	equivalent to a face of some normal Gorenstein Fano polytope.  In the
	present paper, it is shown that, by giving new classes of normal
	Gorenstein Fano polytopes, each order polytope as well as each chain
	polytope of dimension $d$ is unimodularly equivalent to a facet of some
	normal Gorenstein Fano polytopes of dimension $d + 1$.  Furthermore,
	investigation on combinatorial properties, especially, Ehrhart
	polynomials and volume of these new polytopes will be achieved.
	Finally, some curious examples of Gorenstein Fano polytopes will be
	discovered.
\end{abstract} 

\maketitle
\section{Introduction}

First of all, fundamental materials on integral polytopes are summarized and the notation employed in the present paper is introduced.

\subsection{Gorenstein Fano polytopes}
Recall that 
an \textit{integral} convex polytope is a convex polytope all of whose vertices have integer coordinates.
An integral convex polytope $\Pc \subset \RR^d$ of dimension $d$ is called
\textit{Gorenstein Fano} if the origin of $\RR^d$ is a unique lattice point
(i.e., integer point)
belonging to the interior of $\Pc$ and its dual polytope 
\[
\Pc^{\vee} := \{ {\bf x} \in \RR^{d} 
\mid \langle {\bf x}, {\bf y} \rangle \le 1 \  
\mathrm{for\ all}  \ {\bf y} \in \Pc \}
\]
is again integral. 
A Gorenstein Fano polytope is also called a \textit{reflexive} polytope. 
In recent years, the study on Gorenstein Fano polytopes has been achieved by many authors.
It is known that Gorenstein Fano polytopes correspond to Gorenstein toric Fano varieties
and, furthermore, they are related with mirror symmetry (e.g., \cite{mirror,Cox}).
There exist only finitely many Gorenstein Fano polytopes up to unimodular equivalence in each dimension (\cite{Lag}),
and all of them are known up to dimension $4$ (\cite{Kre}).

\subsection{Two poset polytopes}
Let $P = \{p_1, \ldots, p_d\}$ denote a finite partially ordered set 
(\textit{poset}, for short). 
A {\em linear extension} of $P$ is a permutation $\sigma = i_1 i_2 \cdots i_d$ 
of $[d] = \{1, 2, \ldots, d\}$ 
which satisfies $i_a < i_b$ if $p_{i_a} < p_{i_b}$ in $P$. 
Stanley \cite{Stanley}
introduced two classes of integral convex polytopes arising from finite posets, 
order polytopes and chain polytopes.
The \textit{order polytope} $\Oc_P$ of $P$
is defined to be the convex polytope consisting of those $(x_1,\ldots,x_d) \in \RR^d$ such that
\begin{enumerate}
	\item $0 \leq x_i \leq 1$ for $1 \leq i \leq d$;
	\item $x_i \geq x_j$ if $p_i \leq p_j$ in $P$.
\end{enumerate}
The \textit{chain polytope} $\Cc_P$ is defined to be the convex polytope consisting of those $(x_1,\ldots,x_d) \in \RR^d$ such that
\begin{enumerate}
	\item $ x_i \geq 0$ for $1 \leq i \leq d$;
	\item $x_{i_1}+ \cdots + x_{i_k} \leq 1$ for every maximal chain $p_{i_1}< \cdots <p_{i_k}$ of $P$.
\end{enumerate}
It then follows that
both order polytopes and chain polytopes 
are integral convex polytopes of dimension $d$.

We say that an integral convex polytope $\Pc \subset \RR^{d}$ of dimension $d$ 
is \textit{normal} if, for each integer $N > 0$ 
and for each ${\bf a} \in N \Pc \cap \ZZ^{d}$, there exist ${\bf a}_1, \ldots, {\bf a}_N \in \Pc \cap \ZZ^{d}$ 
for which 
${\bf a} = {\bf a}_1 + \cdots +{\bf a}_N$, where 
$N \Pc = \{ N \alpha \mid \alpha \in \Pc \}$.

Let $i(\Pc,n)$ denote the \textit{Ehrhart polynomial} of $\Pc$.
Thus $i(\Pc,n)$ is the numerical function 
\[
i(\Pc, n) := \left| (n\Pc \cap \ZZ^{d}) \right|, 
\, \, \, \, \, 1 \leq n \in \ZZ.
\]
Then $i(\Pc, n)$ is, in fact, a polynomial in $n$ of degree $d$ with $i(\Pc, 0) = 1$ (\cite{Ehrhart}). 
Moreover, the leading coefficient of $i(\Pc,n)$ equals the volume of $\Pc$.

It is known that $\Oc_P$ and $\Cc_P$ are normal with $i(\Oc_P, n)=i(\Cc_P, n)$.
Thus, in particular, 
one has 
$\vol (\Oc_P)=\vol(\Cc_P)=e(P)/d!$, where $e(P)$ is the number of linear extensions of $P$ (\cite[Corollary 4.2]{Stanley}).

\subsection{Gorenstein Fano polytopes arising from poset polytopes}

Given two integral convex polytopes $\Pc$ and $\Qc$ of dimension $d$ in $\RR^d$, 
we set
$$\Gamma(\Pc,\Qc)=\textnormal{conv}\{\Pc\cup (-\Qc)\}\subset \RR^{d}.$$
Let 
$P=\{p_1,\ldots,p_d\}$ and $Q=\{q_1,\ldots,q_d\}$ be finite posets.
Our research objects are $\Gamma(
\Oc_P,
\Oc_Q
)$, $\Gamma(
\Oc_P,
\Cc_Q)$ and $\Gamma(
\Cc_P
,
\Cc_Q)$.
If $P$ and $Q$ possess a common linear extension, 
then $\Gamma(\Oc_P,\Oc_Q)$ is normal Gorenstein Fano (\cite{twin}).
Furthermore,  
each of $\Gamma(\Oc_P,\Cc_Q)$ and $\Gamma(\Cc_P,\Cc_Q)$ is a normal Gorenstein Fano polytope (\cite{orderchain,harmony}).
In addition, if $P$ and $Q$ possess a common linear extension, then these three polytopes have the same Ehrhart polynomial (\cite{volOC})
and a formula to compute their volume is given in terms of the number of linear extensions of the underlying finite posets $P$ and $Q$ (\cite{vol}).

\subsection{Motivation and results}
It is known \cite{refdim} that every integral convex polytope is unimodularly equivalent to a face of some Gorenstein Fano polytope.
This fact naturally lead us to the study of the following question:
\begin{Question}
	\label{q1}
	Is every normal polytope unimodularly equivalent to a face of some normal Gorenstein Fano polytope?
\end{Question}
In this paper, we discuss this question for $(0,1)$-polytopes.
Given two integral convex polytopes $\Pc$ and $\Qc$ of dimension $d$ in $\RR^d$, 
we set
$$\Omega(\Pc,\Qc)=\textnormal{conv}\{(\Pc\times\{1\}) \cup (-\Qc\times\{-1\})\}\subset \RR^{d+1}.$$
Then $\Omega(\Pc,\Qc)$ is an integral convex polytope of dimension $d+1$
and, in addition, each of $\Pc$ and $\Qc$ is a facet of $\Omega(\Pc,\Qc)$.
As an analogy of Question \ref{q1}, we propose the following question:
\begin{Question}
	\label{q2}
	Given any normal $(0,1)$-polytope $\Pc \subset \RR^d$ of dimension $d$, 
	does there exist a normal $(0,1)$-polytope $\Qc \subset \RR^d$ of dimension $d$
	such that $\Omega(\Pc,\Qc)$ is a normal Gorenstein Fano polytope?
\end{Question}
In Section 2, we consider Question \ref{q2} 
for order and chain polytopes.
We discuss the problem when each of $\Omega(
\Oc_P,
\Oc_Q
)$, $\Omega(
\Oc_P,
\Cc_Q)$
and $\Omega(
\Cc_P
,
\Cc_Q)$ is a normal Gorenstein Fano polytope.
Our first main result is
\begin{Theorem}
	\label{Gor}
	Let $P=\{p_1,\ldots,p_d\}$ and $Q=\{q_1,\ldots,q_d\}$ be finite posets.\\
	\textnormal{(1)} If $P$ and $Q$ possess a common linear extension, then $\Omega(
	\Oc_P,
	\Oc_Q
	)$ is a normal Gorenstein Fano polytope.\\
	\textnormal{(2)} Each of $\Omega(
	\Oc_P,
	\Cc_Q)$ and  $\Omega(
	\Cc_P
	,
	\Cc_Q)$ is a normal Gorenstein Fano polytope.
\end{Theorem}

\smallskip
In Section 3, we consider combinatorial properties of these polytopes, especially, 
the Ehrhart polynomials and the volume of $\Omega(
\Oc_P,
\Oc_Q
)$, $\Omega(
\Oc_P,
\Cc_Q)$
and $\Omega(
\Cc_P
,
\Cc_Q)$.
The \textit{ordinal sum} of $P$ and $Q$ is the finite poset $P \oplus Q$ 
on the union $P \cup Q$ such that $s \leq t$ in $P \oplus Q$  
if (a) $s,t \in P$ and $s \leq t$ in $P$,
or (b) $s,t \in Q$ and $s \leq t$ in $Q$,
or (c) $s \in P$ and $t \in Q$.
Our second main result is
\begin{Theorem}
	\label{Ehr}
	Let $P=\{p_1,\ldots,p_d\}$ and $Q=\{q_1,\ldots,q_d\}$ be finite
	poset.  We set $P'=\{p_{d+1}\} \oplus P$ and $Q'=\{q_{d+1}\} \oplus Q$.
	If $P$ and $Q$ possess a common linear extension, then all of 
	$\Omega(
	\Oc_P,
	\Oc_Q
	)$, $\Omega(
	\Oc_P,
	\Cc_Q)$, $\Omega(
	\Cc_P
	,
	\Cc_Q)$, 
	$\Gamma (\Oc_{P'},\Oc_{Q'})$, $\Gamma (\Oc_{P'},\Cc_{Q'})$ and
	$\Gamma(\Cc_{P'},\Cc_{Q'})$ have the same Ehrhart polynomial.
	In particular, these polytopes have the same volume.
\end{Theorem}

Moreover, a combinatorial formula to compute
the volume of these polytopes in terms of the underlying finite posets
$P$ and $Q$ (Theorem \ref{vol}).

Finally, in Section 4, 
some curious examples of Gorenstein Fano polytopes will be
discovered.
It will be shown that there exist normal polytopes $\Pc$ and $\Qc$ such that $\Omega(\Pc,\Qc)$
is Gorenstein Fano, but it is not normal (Example \ref{nonnormal}).
Considering this example, we cannot escape from the temptation to study
Question \ref{q2}.
We also consider a difference of the class of $\Gamma(\Pc, \Qc)$ and that 
of $\Omega(\Pc, \Qc)$.
It is known that the class of $\Omega(
\Oc_P,
\Oc_Q
)$ is included in 
that of $\Gamma(
\Oc_P,
\Oc_Q
)$.
However, the class of  $\Omega(
\Cc_P
,
\Cc_Q)$ is not included in that of  $\Gamma(
\Oc_P,
\Oc_Q
)$,  $\Gamma(
\Oc_P,
\Cc_Q)$ and $\Gamma(
\Cc_P
,
\Cc_Q)$ (Example \ref{class}).
The class of $\Omega(
\Oc_P
,
\Cc_Q)$ is included in none of the above classes.
This fact says that the class of $\Omega(
\Oc_P,
\Cc_Q)$ and that of $\Omega(
\Cc_P
,
\Cc_Q)$ are new classes of normal Gorenstein Fano polytopes.
Futhermore, it will be shown that,
by using these five classes,
there exist 11 normal Gorenstein Fano polytopes of dimension $7$ such that these polytopes have the same Ehrhart polynomial and these polytopes are not unimodularly equivalent each other (Example \ref{five}).

\section{Squarefree Quadratic Gr\"{o}bner basis}
In this section,  we show Theorem \ref{Gor}.

Before proving this theorem, 
we recall some terminologies of finite posets. 
Let $P=\{p_1,\ldots,p_d\}$ be a finite poset.
A subset $I$ of $P$ is called a {\em poset ideal} of $P$ if $p_{i} \in I$ and $p_{j} \in P$ together with $p_{j} \leq p_{i}$ guarantee $p_{j} \in I$.  
Note that the empty set $\emptyset$ and itself $P$ are poset ideals of $P$. 
Let $\Jc(P)$ denote the set of poset ideals of $P$.
A subset $A$ of $P$ is called an {\em antichain} of $P$ if
$p_{i}$ and $p_{j}$ belonging to $A$ with $i \neq j$ are incomparable.  
In particular, the empty set $\emptyset$ and each 1-elemant subsets $\{p_j\}$ are antichains of $P$.
Let $\Ac(P)$ denote the set of antichains of $P$.
For a poset ideal $I$ of $P$, we write $\max(I)$ for the set of maximal elements of $I$.
In particular, $\max(I)$ is an antichain.  
For each subset $I \subset P$, 
we define the $(0, 1)$-vectors $\rho(I) = \sum_{p_{i}\in I} \eb_{i}$, 
where $\eb_{1}, \ldots, \eb_{d}$ are the canonical unit coordinate vectors of $\RR^{d}$.  
In particular $\rho(\emptyset)$ is the origin ${\bf 0}$ of $\RR^{d}$. 
In \cite{Stanley}, it is shown that
\[
\{ {\rm the\ sets\ of\ vertices\ of\ } 
\Oc_P \} = \{ \rho(I) \mid I \in \Jc(P) \}, 
\]
\[
\{ {\rm the\ sets\ of\ vertices\ of\ } 
\Cc_P \} = \{ \rho(A) \mid A \in \Ac(P) \}.
\]


Next, we define the toric rings of integral convex polytopes.
Let $K[{\bf t^{\pm1}}, s] 
= K[t_{1}^{\pm1}, \ldots, t_{d}^{\pm1}, s]$
the Laurent polynomial ring in $d+1$ variables over a field $K$. 
If $\alpha = (\alpha_{1}, \ldots, \alpha_{d}) \in \ZZ^{d}$, then
${\bf t}^{\alpha}s$ is the Laurent monomial
$t_{1}^{\alpha_{1}} \cdots t_{d}^{\alpha_{d}}s \in K[{\bf t^{\pm1}}, s]$. 
In particular ${\bf t}^{\bf 0}s = s$.
Let  $\Pc \subset \RR^{d}$ be an integral convex polytope of dimension $d$ and $\Pc \cap \ZZ^d=\{\ab_1,\ldots,\ab_n\}$.
Then, the \textit{toric ring}
of $\Pc$ is the subalgebra $K[\Pc]$ of $K[{\bf t^{\pm1}}, s] $
generated by
$\{\tb^{\ab_1}s ,\ldots,\tb^{\ab_n}s \}$ over $K$.
We regard $K[\Pc]$ as a homogeneous algebra by setting each $\text{deg } \tb^{\ab_i}s=1$.
The \textit{toric ideal} $I_{\Pc}$ of $\Pc$ is the kernel of a surjective homomorphism $\pi : K[x_1,\ldots,x_n] \rightarrow K[\Pc]$
defined by $\pi(x_i)=\tb^{\ab_i}s$ for $1 \leq i \leq n$.
It is known that $I_{\Pc}$ is generated by homogeneous binomials.
See, e.g.,  \cite{Sturmfels}.

In order to prove Theorem \ref{Gor}, we use the following lemma.
\begin{Lemma}[{\cite[Lemma 1.1]{HMOS}}]
	\label{HMOS}
	Let $\Pc \subset \RR^d$ be an integral convex polytope such that the origin of $\RR^d$ is contained 
	in its interior and $\Pc \cap \ZZ^d=\{\ab_1,\ldots,\ab_n \}$.
	Suppose that any integer point in $\ZZ^{d+1}$ is a linear integer combination of the integer points in $\Pc \times \{1\}$ and there exists an ordering of the variables $x_{i_1} < \cdots < x_{i_n}$ for which $\ab_{i_1}= \mathbf{0}$ such that the initial ideal $\textnormal{in}_{<}(I_{\Pc})$ of the toric ideal $I_{\Pc}$ with respect to the reverse lexicographic order $<$ on the polynomial ring $K[x_1,\ldots,x_n]$
	induced by the ordering is squarefree.
	Then $\Pc$ is a normal Gorenstein Fano.
\end{Lemma}

\smallskip
Now, for finite posets $P=\{p_1,\ldots,p_d\}$ and $Q=\{q_1,\ldots,q_d\}$, let 
\begin{eqnarray*}
	K[{\Oc \Oc}] &=& K[\{x_{I}\}_{I \in \Jc(P)} \cup 
	\{y_{J}\}_{J \in \Jc(Q)} \cup \{ z \}], \\
	K[{\Oc \Cc}] &=& K[\{x_{I}\}_{I \in \Jc(P)} \cup 
	\{y_{\max(J)}\}_{J \in \Jc(Q)} \cup \{ z \}], \\
	K[{\Cc \Cc}] &=& K[\{x_{\max(I)}\}_{I \in \Jc(P)} \cup 
	\{y_{\max(J)}\}_{J \in \Jc(Q)} \cup \{ z \}]
\end{eqnarray*}
denote the polynomial rings over $K$, 
and define the surjective ring homomorphisms 
$\pi_{\Oc \Oc}$, $\pi_{\Oc \Cc}$ and $\pi_{\Cc \Cc}$ by the following: 
\begin{itemize}
	\item $\pi_{\Oc \Oc} : K[{\Oc \Oc}] \to K[\Omega(
	\Oc_P,
	\Oc_Q
	)]$ by setting \\
	$\pi_{\Oc \Oc}(x_{I}) = {\bf t}^{\rho(I\cup\{d+1\})}s$, $\pi_{\Oc \Oc}(y_{J}) = {\bf t}^{- \rho(J\cup\{d+1\})}s$ and 
	$\pi_{\Oc \Oc}(z) = s$, 
	\item $\pi_{\Oc \Cc} : K[{\Oc \Cc}] \to K[\Omega(
	\Oc_P, 
	\Cc_Q)]$ by setting \\
	$\pi_{\Oc \Cc}(x_{I}) = {\bf t}^{\rho(I\cup\{d+1\})}s$, $\pi_{\Oc \Cc}(y_{\max(J)}) = {\bf t}^{- \rho(\max(J)\cup\{d+1\})}s$ and 
	$\pi_{\Oc \Cc}(z) = s$, 
	\item $\pi_{\Cc \Cc} : K[{\Cc \Cc}] \to K[\Omega(
	\Cc_P
	,
	\Cc_Q)]$ by setting \\
	$\pi_{\Cc \Cc}(x_{\max(I)}) = {\bf t}^{\rho(\max(I)\cup\{d+1\})}s$, $\pi_{\Cc \Cc}(y_{\max(J)}) = {\bf t}^{- \rho(\max(J)\cup\{d+1\})}s$ and $\pi_{\Cc \Cc}(z) = s$ 
\end{itemize}
where $I \in \Jc(P)$ and $J \in \Jc(Q)$. 
Then the toric ideal $I_{\Omega(
	\Oc_P,
	\Oc_Q
	)}$ of $\Omega(
\Oc_P,
\Oc_Q
)$ is 
the kernel of $\pi_{\Oc \Oc}$. 
Similarly, the toric ideal $I_{\Omega(
	\Oc_P,
	\Cc_Q)}$ (resp. $I_{\Omega(
	\Cc_P
	,
	\Cc_Q)}$) is 
the kernel of $\pi_{\Oc \Cc}$ (resp. $\pi_{\Cc \Cc}$). 


Next, we introduce monomial orders $<_{\Oc \Oc}$, $<_{\Oc \Cc}$ and $<_{\Cc \Cc}$
and $\Gc_{\Oc \Oc}$, $\Gc_{\Oc \Cc}$ and $\Gc_{\Cc \Cc}$ which are the set of binomials. 
Let $<_{\Oc \Oc}$ denote a reverse lexicographic order on $K[{\Oc \Oc}]$
satisfying
\begin{itemize}
	\item
	$z <_{\Oc \Oc}  y_{J} <_{\Oc \Oc} x_{I}$;
	\item
	$x_{I'} <_{\Oc \Oc} x_{I}$ if $I' \subset I$;
	\item
	$y_{J'} <_{\Oc \Oc} y_{J}$ if $J' \subset J$,
\end{itemize}
and $\Gc_{\Oc \Oc} \subset K[{\Oc \Oc}]$ the set of the following binomials:
\begin{enumerate}
	\item[(i)]
	$x_{I}x_{I'} - x_{I\cup I'}x_{I \cap I'}$;
	\item[(ii)]
	$y_{J}y_{J'} - y_{J\cup J'}y_{J \cap J}$;
	\item[(iii)]
	$x_{I}y_{J} - x_{I \setminus \{p_{i}\}}y_{J \setminus \{q_{i}\}}$;
	\item[(iv)]
	$x_{\emptyset}y_{\emptyset} - z^2$,
\end{enumerate}
and let $<_{\Oc \Cc}$ denote a reverse lexicographic order on $K[{\Oc \Cc}]$
satisfying
\begin{itemize}
	\item
	$z <_{\Oc \Cc}  y_{\max(J)} <_{\Oc \Cc} x_{I}$;
	\item
	$x_{I'} <_{\Oc \Cc} x_{I}$ if $I' \subset I$;
	\item
	$y_{\max(J')} <_{\Oc \Cc} y_{\max(J)}$ if $J' \subset J$,
\end{itemize}
and $\Gc_{\Oc \Cc} \subset K[{\Oc \Cc}]$ the set of the following binomials:
\begin{enumerate}
	\item[(v)]
	$x_{I}x_{I'} - x_{I\cup I'}x_{I \cap I'}$;
	\item[(vi)]
	$y_{\max(J)}y_{\max(J')} - y_{\max(J\cup J')}y_{\max(J * J')}$;
	\item[(vii)]
	$x_{I}y_{\max(J)} - x_{I \setminus \{p_{i}\}}y_{\max(J) \setminus \{q_{i}\}}$;
	\item[(viii)]
	$x_{\emptyset}y_{\emptyset} - z^2$,
\end{enumerate}
and let $<_{\Cc \Cc}$ denote a reverse lexicographic order on $K[{\Cc \Cc}]$
satisfying
\begin{itemize}
	\item
	$z <_{\Cc \Cc}  y_{\max(J)} <_{\Cc \Cc} x_{\max(I)}$;
	\item
	$x_{\max(I')} <_{\Cc \Cc} x_{\max(I)}$ if $I' \subset I$;
	\item
	$y_{\max(J')} <_{\Cc \Cc} y_{\max(J)}$ if $J' \subset J$,
\end{itemize}
and $\Gc_{\Cc \Cc} \subset K[{\Cc \Cc}]$ the set of the following binomials:
\begin{enumerate}
	\item[(ix)]
	$x_{\max(I)}x_{\max(I')} - y_{\max( I \cup I')}y_{\max(I * I')}$;
	\item[(x)]
	$y_{\max(J)}y_{\max(J')} - y_{\max(J\cup J')}y_{\max(J * J')}$;
	\item[(xi)]
	$x_{\max(I)}y_{\max(J)} - x_{\max(I) \setminus \{p_{i}\}}y_{\max(J) \setminus \{q_{i}\}}$;
	\item[(xii)]
	$x_{\emptyset}y_{\emptyset} - z^2$,
\end{enumerate}
where 
\begin{itemize}
	\item
	$I$ and $I'$ are poset ideals of $P$ which are incomparable in $\Jc(P)$;
	\item  
	$J$ and $J'$ are poset ideals of $Q$ which are incomparable in $\Jc(Q)$;
	\item
	$I * I'$ is the poset ideal of $P$ generated by $\max(I \cap I')\cap (\max(I)\cup \max(I'))$;
	\item
	$J * J'$ is the poset ideal of $Q$ generated by $\max(J \cap J')\cap (\max(J)\cup \max(J'))$;
	\item $p_{i}$ is a maximal element of $I$ and $q_{i}$ is a maximal element of $J$. 
\end{itemize}

\begin{Proposition}
	\label{OO}
	Work with the same situation as above.
	If $P$ and $Q$ possess a common linear extension,
	then the origin of $\RR^{d+1}$ is contained in the interior of $\Omega(
	\Oc_P,
	\Oc_Q
	)$
	and  $\Gc_{OO}$ is a Gr\"{o}bner basis of $I_{\Omega(
		\Oc_P,
		\Oc_Q
		)}$ with respect to $<_{OO}.$
\end{Proposition}
\begin{proof}	
	Set $P'=\{p_{d+1}\} \oplus P$ and $Q'=\{q_{d+1}\} \oplus Q$.
	Then we have 
	$$\Jc(P')=\{\emptyset\} \cup \{I \cup \{p_{d+1}\} | I \in \Jc(P) \},$$
	$$\Jc(Q')=\{\emptyset\} \cup \{J \cup \{q_{d+1}\} | J \in \Jc(Q) \}.$$
	Hence we know that $\Omega(
	\Oc_P,
	\Cc_Q
	)=\Gamma(\Oc_{P'},\Oc_{Q'})$.
	By \cite{twin}, we can easily show 	if $P$ and $Q$ possess a common linear extension,
	then the origin of $\RR^{d+1}$ is contained in the interior of $\Omega(
	\Oc_P,
	\Oc_Q
	)$
	and  $\Gc_{OO}$ is a Gr\"{o}bner basis of $I_{\Omega(
		\Oc_P,
		\Oc_Q
		)}$ with respect to $<_{OO}$, as desired.
\end{proof}

\begin{Proposition}
	\label{OC}
	Work with the same situation as above.
	Then $\Gc_{OC}$ is a Gr\"{o}bner basis of $I_{\Omega(
		\Oc_P,
		\Cc_Q)}$ with respect to $<_{OC}.$
\end{Proposition}

\begin{proof}
	It is clear that $\Gc_{\Oc\Cc} \subset I_{\Omega(
		\Oc_P, 
		\Cc_Q)}$.
	For a binomial $f = u - v$, $u$ is called the {\em first}
	monomial of $f$ and $v$ is called the {\em second} monomial of $f$. 
	We note that the initial monomial of each of the binomials (v) -- (viii) 
	with respect to $<_{\Oc\Cc}$ is its first monomial. 
	Let ${\rm in}_{<_{\Oc\Cc}}(\Gc_{\Oc\Cc})$ denote the set of initial monomials of binomials 
	belonging to $\Gc_{\Oc\Cc}$.  It follows from \cite[(0.1)]{OHrootsystem} that,
	in order to show that $\Gc_{\Oc\Cc}$ is a Gr\"obner basis of
	$I_{\Omega(
		\Oc_P,
		\Cc_Q)}$ with respect to $<_{\Oc\Cc}$, we must prove 
	the following assertion:
	($\clubsuit$) If $u$ and $v$ are monomials belonging to 
	$K[\Oc \Cc]$ with $u \neq v$ such that 
	$u \not\in \langle {\rm in}_{<_{\Oc\Cc}}(\Gc_{\Oc\Cc}) \rangle$ 
	and $v \not\in \langle {\rm in}_{<_{\Oc\Cc}}(\Gc_{\Oc\Cc}) \rangle$,
	then $\pi_{\Oc\Cc}(u) \neq \pi_{\Oc\Cc}(v)$.
	
	Let $u, v \in K[\Oc\Cc]$ be monomials
	with $u \neq v$.  Write
	\[
	u = z^{\alpha} x_{I_{1}}^{\xi_{1}} \cdots x_{I_{a}}^{\xi_{a}}
	y_{\max(J_{1})}^{\nu_{1}} \cdots y_{\max(J_{b})}^{\nu_{b}},
	\, \, \, \, \, \, \, \, \, \, 
	v = z^{\alpha'} x_{I'_{1}}^{\xi'_{1}} \cdots x_{I'_{a'}}^{\xi'_{a'}}
	y_{\max(J'_{1})}^{\nu'_{1}} \cdots y_{\max(J'_{b'})}^{\nu'_{b'}},
	\]
	where
	\begin{itemize}
		\item
		$\alpha \geq 0$, $\alpha' \geq 0$;
		\item
		$I_{1}, \ldots, I_{a}, I'_{1}, \ldots, I'_{a'} 
		\in \Jc(P)$;
		\item
		$J_{1}, \ldots, J_{b}, J'_{1}, \ldots, J'_{b'} 
		\in \Jc(Q)$;
		\item
		$\xi_{1}, \ldots, \xi_{a}, 
		\nu_{1}, \ldots, \nu_{b},
		\xi'_{1}, \ldots, \xi'_{a'}, 
		\nu'_{1}, \ldots, \nu'_{b'} > 0$,
	\end{itemize}
	and where $u$ and $v$ are relatively prime with
	$u \not\in \langle {\rm in}_{<_{\Oc\Cc}}(\Gc_{\Oc\Cc}) \rangle$ 
	and $v \not\in \langle {\rm in}_{<_{\Oc\Cc}}(\Gc_{\Oc\Cc}) \rangle$.
	Thus
	By using (v) and (vi), it follows that
	\begin{itemize}
		\item
		$I_{1} \subsetneq  I_{2} \subsetneq \cdots \subsetneq I_{a}$ and $J_{1} \subsetneq J_{2} \subsetneq \cdots \subsetneq J_{b}$;
		\item
		$I'_{1} \subsetneq I'_{2} \subsetneq \cdots \subsetneq I'_{a'}$ and	$J'_{1} \subsetneq J'_{2} \subsetneq \cdots \subsetneq J'_{b'}$.
	\end{itemize}
	
	Now, suppose that $\pi_{\Oc\Cc}(u)=\pi_{\Oc\Cc}(v)$.
	Then we have 
	$$\sum\limits_{\stackrel{I \in \{I_1,\ldots,I_a\}}{p_i \in I}}\xi_I-\sum\limits_{\stackrel{J \in \{J_1,\ldots,J_b\}}{q_i \in \max(J)}}\nu_J
	=\sum\limits_{\stackrel{I' \in \{I'_1,\ldots,I'_{a'}\}}{p_i \in I'}}\xi'_{I'}-\sum\limits_{\stackrel{J' \in \{J'_1,\ldots,J'_{b'}\}}{q_i \in \max(J')}}\nu'_{J'}.$$
	for all $1 \leq i \leq d$ by comparing the degree of $t_i$. 
	
	Assume that $(a,a')\neq (0,0)$ and $I_a \setminus I'_{a'} \neq \emptyset$.
	Then there exists a maximal element $p_{i^*}$ of $I_a$ with $p_{i^*} \notin I'_{a'}$. 
	Since $p_{i^*} \notin I'_{a'}$, one has
	$$\sum\limits_{\stackrel{I \in \{I_1,\ldots,I_a\}}{p_{i^*} \in I}}\xi_I-\sum\limits_{\stackrel{J \in \{J_1,\ldots,J_b\}}{q_{i^*} \in \max(J)}}\nu_J
	=-\sum\limits_{\stackrel{J' \in \{J'_1,\ldots,J'_{b'}\}}{q_{i^*} \in \max(J')}}\nu'_{J'}\leq 0.$$
	Moreover,  since $p_{i^*}$ is belonging to $I_a$, we also have 
	$$\sum\limits_{\stackrel{I \in \{I_1,\ldots,I_a\}}{p_{i^*} \in I}}\xi_I>0. $$
	Hence there exists an integer $c$ with $1\leq c \leq b$ such that $q_{i^*}$ is a maximal element of $J_c$.
	Therefore we have $x_{I_a}y_{\max(J_c)} \in \langle{\rm in}_{<_{\Oc\Cc}}(\Gc_{\Oc\Cc}) \rangle$, but this is a contradiction.
	By considering the case where  $(a,a')\neq (0,0)$ and $I'_{a'} \setminus I_{a} \neq \emptyset$, it is known that one of  the followings is satisfied:
	\begin{itemize}
		\item $(a,a')=(1,0), I_a=\emptyset$;
		\item $(a,a')=(0,1), I_{a'}=\emptyset$;
		\item $(a,a')=(0,0)$.
	\end{itemize}
	Then we have 
	$$\sum\limits_{\stackrel{J \in \{J_1,\ldots,J_b\}}{q_i \in \max(J)}}\nu_J
	=\sum\limits_{\stackrel{J' \in \{J'_1,\ldots,J'_{b'}\}}{q_i \in \max(J')}}\nu'_{J'}.$$
	for all $1 \leq i \leq d$.
	Assume that $(b,b')\neq (0,0)$ and $J_b \setminus J'_{b'} \neq \emptyset$.
	Then there exists a maximal element $q_{i'}$ of $J_b$ with $q_{i'} \notin J'_{b'}$. 
	Since $q_{i'} \notin J'_{b'}$, one has
	$$0<\sum\limits_{\stackrel{J \in \{J_1,\ldots,J_b\}}{q_{i'} \in \max(J)}}\nu_J
	\neq \sum\limits_{\stackrel{J' \in \{J'_1,\ldots,J'_{b'}\}}{q_{i'} \in \max(J')}}\nu'_{J'}=0,$$
	but this is a contradiction.
	By considering the case where  $(b,b')\neq (0,0)$ and $J'_{b'} \setminus J_{b} \neq \emptyset$, it is known that one of  the followings is satisfied:
	\begin{itemize}
		\item $(b,b')=(1,0), J_b=\emptyset$;
		\item $(b,b')=(0,1), J'_{b'}=\emptyset$;
		\item $(b,b')=(0,0)$.
	\end{itemize}
	Hence one has
	$u = z^{\alpha} x_{{\emptyset}}^{\xi}y_{\emptyset}^{\nu}$ and
	$v = z^{\alpha'} x_{\emptyset}^{\xi'} y_{\emptyset}^{\nu'}$,
	where $\xi,\xi',\nu,\nu' \geq 0$.
	Since $x_{\emptyset}y_{\emptyset} \in  \langle{\rm in}_{<_{\Oc\Cc}}(\Gc_{\Oc\Cc}) \rangle$ and since $u$ and $v$ are relatively prime,
	we may assume that $\nu=\xi'=0$.
	Thus $u = z^{\alpha} x_{{\emptyset}}^{\xi}$ and
	$v = z^{\alpha'} y_{\emptyset}^{\nu'}$.
	Note that either $\alpha = 0$ or $\alpha' = 0$.
	Hence by comparing the degree of $t^{d+1}$, it is known that $\xi=\nu'=\alpha=\alpha'=0$, contradiction.
\end{proof}

\begin{Proposition}
	\label{CC}
	Work with the same situation as above.
	Then $\Gc_{CC}$ is a Gr\"{o}bner basis of $I_{\Omega(
		\Cc_P
		,
		\Cc_Q)}$ with respect to $<_{CC}.$
\end{Proposition}

\begin{proof}
	We can show that the assertion follows by a similar way in the proof of Proposition \ref{OC}.
\end{proof}

Finally, we show Theorem \ref{Gor}.
\begin{proof}[Proof of Theorem \ref{Gor}]
	It is easy to show that any integer point in $\ZZ^{d+2}$ is a linear integer combination of the integer points in $\Omega(\Oc_{P}, \Oc_{Q}) \times \{1\}$ (resp. $\Omega(\Oc_{P}, \Cc_{Q}) \times \{1\}$ and $\Omega(\Cc_{P}, \Cc_{Q}) \times \{1\}$). 
	By Lemma \ref{HMOS} and Proposition \ref{OO}, \ref{OC}, \ref{CC}, the assertion follows.	
\end{proof}

\section{Ehrhart polynomials and volume}
In this section, we consider combinatorial properties of these polytopes, especially, 
the Ehrhart polynomials and the volume of $\Omega(
\Oc_P,
\Cc_Q
)$, $\Omega(
\Oc_P,
\Cc_Q
)$ and $\Omega(
\Oc_P,
\Cc_Q
)$,
for finite posets $P=\{p_1,\ldots,p_d\}$ and $Q=\{q_1,\ldots,q_d\}$.
In particular, we show Theorem \ref{Ehr}.

Let $\Pc \subset \RR^d$ be an integral convex polytope of dimension $d$.
In order to prove Theorem \ref{Ehr}, we use the following facts. 
\begin{itemize}
	\item  
	If $\Pc$ is normal, then the Ehrhart polynomial of $\Pc$ is equal to the Hilbert polynomial of 
	the toric ring $K[\Pc]$.  \\
	\item Let $S$ be a polynomial ring and $I \subset S$ be a graded ideal of $S$. 
	Let $<$ be a monomial order on $S$. 
	Then $S/I$ and $S/\mathrm{in}_{<}(I)$ have the same Hilbert function. 
	(see \cite[Corollary 6.1.5]{Monomial})
\end{itemize}

Here, we put
\[
R_{\Oc \Oc} := K[\Oc \Oc]/\mathrm{in}_{<_{\Oc \Oc}}(I_{\Omega(
	\Oc_P, 
	\Oc_Q
	)}), 
\]
\[
R_{\Oc \Cc} := K[\Oc \Cc]/\mathrm{in}_{<_{\Oc \Cc}}(I_{\Omega(
	\Oc_P, 
	\Cc_Q)}), 
\]
\[
R_{\Cc \Cc} := K[\Cc \Cc]/\mathrm{in}_{<_{\Cc \Cc}}(I_{\Omega(
	\Cc_P
	, 
	\Cc_Q)}). 
\]
\begin{Proposition}
	\label{iso}
	Work with the same situation as above.
	If $P$ and $Q$ possess a common linear extension, then these rings 
	$R_{\Oc \Oc}$, $R_{\Oc \Cc}$ and $R_{\Cc \Cc}$ are isomorphic. 
\end{Proposition}

\begin{proof}
	By Proposition \ref{OO}, \ref{OC} and \ref{CC}, we have
	\[
	R_{\Oc \Oc} \cong \frac{K[\{x_{I}\}_{ I \in \Jc(P)} \cup 
		\{y_{J}\}_{J \in \Jc(Q)} \cup \{ z \}]}
	{(x_{I}x_{I'}, y_{J}y_{J'}, x_{I}y_{J},x_{\emptyset}y_{\emptyset}  \mid I, I', J \ \mathrm{and}\ J' \ \mathrm{satisfy \ (*)}) }, 
	\]
	\[
	R_{\Oc \Cc} \cong \frac{K[\{x_{I}\}_{I \in \Jc(P)} \cup 
		\{y_{\max(J)}\}_{J \in \Jc(Q)} \cup \{ z \}]}
	{(x_{I}x_{I'}, y_{\max(J)}y_{\max(J')}, x_{I}y_{\max(J)},x_{\emptyset}y_{\emptyset} \mid I, I', J \ \mathrm{and}\ J' \ \mathrm{satisfy \ (*)}) }, 
	\]
	\[
	R_{\Cc \Cc} \cong \frac{K[\{x_{\max(I)}\}_{I \in \Jc(P)} \cup 
		\{y_{\max(J)}\}_{J \in \Jc(Q)} \cup \{ z \}]}
	{(x_{\max(I)}x_{\max(I')}, y_{\max(J)}y_{\max(J')}, x_{\max(I)}y_{\max(J)},x_{\emptyset}y_{\emptyset}  \mid I, I', J \ \mathrm{and}\ J' \ \mathrm{satisfy \ (*)}) }, 
	\]
	where the condition $(*)$ is the following: 
	\begin{itemize}
		\item
		$I$ and $I'$ are poset ideals of $P$ which are incomparable in $\Jc(P)$;
		\item  
		$J$ and $J'$ are poset ideals of $Q$ which are incomparable in $\Jc(Q)$;
		\item 
		There exists $1 \leq i \leq d$ such that $p_{i}$ is a maximal element of $I$ and $q_{i}$ is a maximal element of $J$. 
	\end{itemize}
	Hence it is easy to see that 
	the ring homomorphism $\phi : R_{\Oc \Cc} \to R_{\Cc \Cc}$ by setting $\phi(x_I) = x_{\max(I)}$, 
	$\phi(y_{\max(J)}) = y_{\max(J)}$ and $\phi(z) = z$ is an isomorphism. 
	Similarly, if $P$ and $Q$ possess a common linear extension, 
	we can see that the ring homomorphism $\phi^{'} : R_{\Oc \Oc} \to R_{\Oc \Cc}$ by setting $\phi^{'}(x_I) = x_I$, 
	$\phi^{'}(y_J) = y_{\max(J)}$ and $\phi^{'}(z) = z$ is an isomorphism.
	Hence it is known that $R_{\Oc \Oc} \cong R_{\Oc \Cc} \cong R_{\Cc \Cc}$, as desired.
\end{proof} 

Now, we prove Theorem \ref{Ehr}.
\begin{proof}[Proof of Theorem \ref{Ehr}]
	By Theorem \ref{Gor}, it is known that
	that  $\Omega(
	\Oc_P,  
	\Oc_Q
	)$, $\Omega(
	\Oc_P,  
	\Cc_Q)$ and $\Omega(
	\Cc_P
	,  
	\Cc_Q)$ are normal. 
	Hence the Ehrhart polynomial of $\Omega(
	\Oc_P,  
	\Oc_Q
	)$ (resp.  $\Omega(
	\Oc_P,  
	\Cc_Q)$ and $\Omega(
	\Cc_P
	,  
	\Cc_Q)$) 
	is equal to the Hilbert polynomial of 
	$K[\Omega(
	\Oc_P,  
	\Oc_Q
	)]$ (resp.  $\Omega(
	\Oc_P,  
	\Cc_Q)$ and $K[\Omega(
	\Cc_P
	, 
	\Cc_Q)]$). 
	By Proposition \ref{iso}, $R_{\Oc \Oc}$, $R_{\Oc \Cc}$ and $R_{\Cc \Cc}$ have the same Hilbert polynomial. 
	Hence $K[\Omega(
	\Oc_P,  
	\Oc_Q
	)]$, $K[\Omega(
	\Oc_P,  
	\Cc_Q)]$ and $K[\Omega(
	\Cc_P
	, 
	\Cc_Q)]$ 
	also have the same Hilbert polynomial. 
	On the other hand, in the proof of Proposition \ref{OC}, it is known that $\Omega(
	\Oc_P,
	\Oc_Q
	)=\Gamma(\Oc(P'),\Oc(Q'))$.
	Hence by \cite[Theorem 1.1]{volOC},  we have the desired conclusion. 
\end{proof}

Finally, we give a combinatorial formula to compute the volume of these polytopes
in terms of the underlying finite posets
$P$ and $Q$.

Given a subset $W$ of $[d]$ we define the \textit{induced subposet} of $P$ on $W$
to be the finite poset $P_W=\{p_i \mid i \in W \}$
such that  $p_i \leq p_j$ in $P_W$ if and only if $p_i \leq p_j$ in $P$. 
For $W \subset [d]$,
we set $\Delta_W(P,Q)=P_W \oplus Q_{\overline{W}}$, where $\overline{W}=[d] \setminus W$.
Note that
$\Delta_{W}(P,Q)$ is a $d$-element poset 
and we have $\Ac(\Delta_W(P,Q))=\Ac(P_W) \cup \Ac(Q_{\overline{W}})$.
Let $W=\{i_1,\ldots,i_k\} \subset [d]$ and $\overline{W}=\{i_{k+1},\ldots,i_d\} \subset [d]$
with $W \cup \overline{W}=[d]$.
Then we have $\Delta_W(P,Q)=\{p_{i_1},\ldots,p_{i_k},q_{i_{k+1}},\ldots,q_{i_d}  \}$.

Now, by \cite[Theorem 1.3]{vol} and Theorem \ref{Ehr}, we obtain the following theorem.
\begin{Theorem}
	\label{vol}
	Let $P=\{p_1,\ldots,p_d\}$ and $Q=\{q_1,\ldots,q_d\}$ be finite posets, 
	and set $P'=\{p_{d+1}\} \oplus P$ and $Q'=\{q_{d+1}\} \oplus Q$.
	If $P$ and $Q$ possess a common linear extension, then we have
	$$\textnormal{vol}(\Omega(
	\Oc_P, 
	\Cc_Q
	))=\sum\limits_{W \subset [d+1]}\cfrac{e(\Delta_W(P',Q'))}{(d+1)!}.$$
\end{Theorem}

\section{Examples}
In this section, we give some curious examples of Gorenstein Fano polytopes.
At first, the following example motivates considering Question \ref{q2}.
\begin{Example}
	\label{nonnormal}
	Let $\Pc \subset \RR^9$ be the $(0,1)$-polytope of dimension $9$ whose vertices are followings:
	$$e_1+e_2,e_2+e_3,e_3+e_4,e_4+e_5,e_1+e_5,e_1+e_6,e_1+e_7,e_2+e_7,e_2+e_8,e_3+e_8,$$
	$$e_3+e_9,e_4+e_9,e_4,e_5,e_5+e_6.$$
	Then $\Pc$ is normal \textnormal{(\cite{Opoly})}.
	Moreover, $\Omega(\Pc,\Pc)$ is Gorenstein Fano, but it is not normal.
\end{Example}
By this example, it is  known that even if $\Pc$ and $\Qc$ are normal $(0,1)$-polytopes, $\Omega(\Pc,\Qc)$ is not always normal.
It hasn't been known whether there exists a normal Gorenstein Fano polytope $\Qc$ such that the normal polytope $\Pc$ in Example \ref{nonnormal} is unimodularly equivalent to a face of $\Qc$ or not.

Next, we consider a difference of the class of $\Gamma(\Pc, \Qc)$ and the class of $\Omega(\Pc, \Qc)$.
It is known that the class of $\Omega(
\Oc_P,
\Oc_Q
)$ is included in 
that of $\Gamma(
\Oc_P,
\Oc_Q
)$.
\begin{Example}
	\label{class}
	Let $P$ be the finite poset  as follows,
	\newline
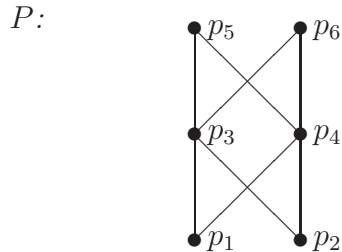
	\begin{picture}(400,150)(10,50)
	\put(100,170){$P$:}
	\put(170,170){\circle*{5}}
	\put(170,130){\circle*{5}}
	\put(210,170){\circle*{5}}
	\put(210,130){\circle*{5}}
	\put(210,90){\circle*{5}}
	\put(170,90){\circle*{5}}
	\put(215,168){$p_6$}
	\put(175,168){$p_5$}
	\put(215,128){$p_4$}
	\put(175,128){$p_3$}
	\put(215,88){$p_2$}
	\put(175,88){$p_1$}
	\put(170,170){\line(0,-1){80}}
	\put(210,170){\line(0,-1){80}}
	\put(170,170){\line(1,-1){40}}
	\put(210,170){\line(-1,-1){40}}
	\put(170,90){\line(1,1){40}}
	\put(210,90){\line(-1,1){40}}
	\end{picture}\\
	For any finite poset $P'$ with $7$ elements, it is known that 
	the $f$-vector of $\Omega(\Cc_P	,	\Cc_P)$ is not equal to that of $\Gamma(\Oc_{P'},\Oc_{P'})$  and $\Gamma(\Cc_{P'},\Cc_{P'})$.
	Hence 	$\Omega(
	\Cc_P
	,
	\Cc_P
	)$ is not unimodularly equivalent to
	$\Gamma(\Oc_{P'},\Oc_{P'})$ and $\Gamma(\Cc_{P'},\Cc_{P'})$.  
\end{Example}
By this example, we know that
the class of  $\Omega(
\Cc_P
,
\Cc_Q)$ is not included in that of  $\Gamma(
\Oc_P,
\Oc_Q
)$,  $\Gamma(
\Oc_P,
\Cc_Q)$ and $\Gamma(
\Cc_P
,
\Cc_Q)$.
Similarly, the class of $\Omega(
\Oc_P
,
\Cc_Q)$ is included in none of the above classes.
This fact says that the class of $\Omega(
\Oc_P,
\Cc_Q)$ and that of $\Omega(
\Cc_P
,
\Cc_Q)$ are new classes of normal Gorenstein Fano polytopes.

\begin{Example}
	\label{five}
	Let $P$ be the finite poset as in Example \ref{class} and $P'=\{p_{7}\}\oplus P$.
	Also, we let $P_1, P_2$ and $P_3$ be the finite posets as follows:
	\newline
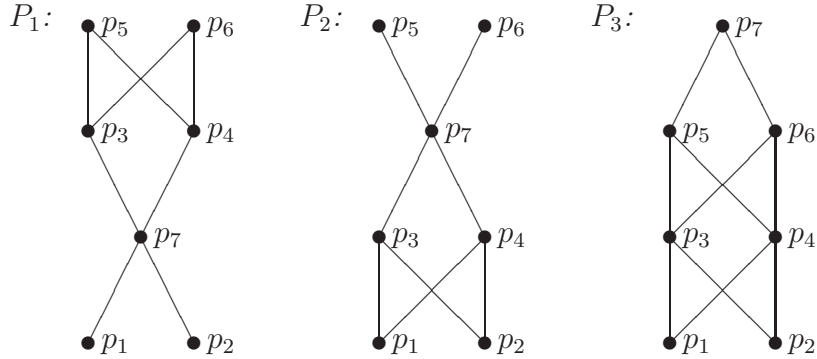
	\begin{picture}(400,200)(10,0)
	\put(60,170){$P_1$:}
	\put(90,170){\circle*{5}}
	\put(90,130){\circle*{5}}
	\put(130,170){\circle*{5}}
	\put(130,130){\circle*{5}}
	\put(130,50){\circle*{5}}
	\put(90,50){\circle*{5}}
	\put(135,168){$p_6$}
	\put(95,168){$p_5$}
	\put(135,128){$p_4$}
	\put(95,128){$p_3$}
	\put(135,48){$p_2$}
	\put(95,48){$p_1$}
	\put(115,88){$p_7$}
	\put(110,90){\circle*{5}}
	\put(90,170){\line(0,-1){40}}
	\put(130,170){\line(0,-1){40}}
	\put(90,170){\line(1,-1){40}}
	\put(130,170){\line(-1,-1){40}}
	\put(90,130){\line(1,-2){40}}
	\put(130,130){\line(-1,-2){40}}

	\put(170,170){$P_2$:}
	\put(200,170){\circle*{5}}
	\put(200,90){\circle*{5}}
	\put(240,170){\circle*{5}}
	\put(240,90){\circle*{5}}
	\put(240,50){\circle*{5}}
	\put(200,50){\circle*{5}}
	\put(245,168){$p_6$}
	\put(205,168){$p_5$}
	\put(245,88){$p_4$}
	\put(205,88){$p_3$}
	\put(245,48){$p_2$}
	\put(205,48){$p_1$}
	\put(225,128){$p_7$}
	\put(220,130){\circle*{5}}
	\put(200,90){\line(0,-1){40}}
	\put(240,90){\line(0,-1){40}}
	\put(200,50){\line(1,1){40}}
	\put(240,50){\line(-1,1){40}}
	\put(200,170){\line(1,-2){40}}
	\put(240,170){\line(-1,-2){40}}
	
	\put(280,170){$P_3$:}
	\put(310,130){\circle*{5}}
	\put(310,90){\circle*{5}}
	\put(350,130){\circle*{5}}
	\put(350,90){\circle*{5}}
	\put(350,50){\circle*{5}}
	\put(310,50){\circle*{5}}
	\put(355,128){$p_6$}
	\put(315,128){$p_5$}
	\put(355,88){$p_4$}
	\put(315,88){$p_3$}
	\put(355,48){$p_2$}
	\put(315,48){$p_1$}
	\put(335,168){$p_7$}
	\put(330,170){\circle*{5}}
	\put(310,130){\line(0,-1){80}}
	\put(350,130){\line(0,-1){80}}
	\put(310,130){\line(1,-1){40}}
	\put(350,130){\line(-1,-1){40}}
	\put(310,50){\line(1,1){40}}
	\put(350,50){\line(-1,1){40}}
	\put(330,170){\line(1,-2){20}}
	\put(330,170){\line(-1,-2){20}}
	\end{picture}\\
	Then we have $$\Gamma(\Cc_{P'},\Cc_{P'})=\Gamma(\Cc_{P_1},\Cc_{P_1})=\Gamma(\Cc_{P_2},\Cc_{P_2})=\Gamma(\Cc_{P_3},\Cc_{P_3}).$$
	Hence it is known that  the 11 normal Gorenstein Fano polytopes $$\Omega(
	\Oc_P,
	\Oc_P), \Omega(
	\Oc_P,
	\Cc_P
	), \Omega(
	\Cc_P
	,
	\Cc_P
	),$$   $$\Gamma(\Oc_{P'},\Cc_{P'}), \Gamma(\Cc_{P'},\Cc_{P'}),$$
	$$\Gamma(\Oc_{P_1},\Oc_{P_1}), \Gamma(\Oc_{P_2},\Oc_{P_2}), \Gamma(\Oc_{P_3},\Oc_{P_3}),$$
	$$\Gamma(\Oc_{P_1},\Cc_{P_1}), \Gamma(\Oc_{P_2},\Cc_{P_2}), \Gamma(\Oc_{P_3},\Cc_{P_3})$$
	have the same Ehrhart polynomial.
	However, these polytopes are not unimodularly equivalent each other.
\end{Example}

By these five classes of normal Gorenstein Fano polytopes, we can obtain several interesting examples.
From this example, one of the future problem is to discuss how many Gorenstein Fano polytopes which have the same Ehrhart polynomial.

Finally, we give some examples of this problem.
\begin{Example}
	Let $\Pc \subset \RR^d$ be the normal Gorenstein Fano simplex of dimension $d$ whose vertices are followings:
	$$e_1,\ldots,e_d,-e_1-\cdots-e_d.$$
	Then we have $i(\Pc,n)=\sum_{i=0}^{d}\binom{n+d-i}{d}$.
	On the other hand, every Gorenstein Fano polytope of dimension $d$ whose Ehrhart polynomial is equal to
	$\sum_{i=0}^{d}\binom{n+d-i}{d}$ is unimodularly equivalemt to $\Pc$.
\end{Example}

\begin{Example}
	By checking any Gorenstein Fano polytopes of dimension $2$,
	we obtain followings:
	\begin{itemize}
		\item The number of Gorenstein Fano polytopes whose Ehrhart polynomials equal $\frac{3}{2}n^2+\frac{3}{2}n+1$ is $1$;
		\item The number of Gorenstein Fano polytopes whose Ehrhart polynomials equal $2n^2+2n+1$ is $3$;
		\item The number of Gorenstein Fano polytopes whose Ehrhart polynomials equal $\frac{3}{2}n^2+\frac{3}{2}n+1$ is $2$;
		\item The number of Gorenstein Fano polytopes whose Ehrhart polynomials equal $3n^2+3n+1$ is $4$;
		\item The number of Gorenstein Fano polytopes whose Ehrhart polynomials equal $\frac{5}{2}n^2+\frac{5}{2}n+1$ is $2$;
		\item The number of Gorenstein Fano polytopes whose Ehrhart polynomials equal $4n^2+4n+1$ is $3$;
		\item The number of Gorenstein Fano polytopes whose Ehrhart polynomials equal $\frac{7}{2}n^2+\frac{7}{2}n+1$ is $1$.
	\end{itemize}
\end{Example}

\end{document}